\documentclass[11pt,a4paper]{article}

\usepackage{enumerate}
\usepackage{a4wide}
\usepackage{amsfonts,amsmath}
\usepackage{latexsym}
\usepackage{xcolor}
\usepackage{amsthm}
\usepackage[colorlinks=true, linkcolor=black]{hyperref}
\usepackage{times}
\author{Pavlos Kalantzopoulos}
\date{}
\title{Commutativity in a Poincar\'e-Type Inequality}

\newtheorem{theorem}{Theorem}

\newcommand{\R}{\mathbb{R}}

\newtheorem*{theoremrep}{Theorem \ref{unconditionalK}}

\newtheorem{lemma}{Lemma}

\newtheorem{prop}[lemma]{Proposition}
\newtheorem{remark}[lemma]{Remark}
\newtheorem{question}[lemma]{Question}

\newtheorem{conjecture}[lemma]{Conjecture}

\begin{document}
	
\maketitle

\begin{abstract}
We study a Poincar\'e-type inequality for Gaussian measures restricted to origin-symmetric convex sets. The inequality arises from the second variation at $t=0$ of the function
\begin{equation*}
t\longmapsto \gamma_{\Sigma}\left(e^{tA}K\right),
\end{equation*}
where $\Sigma$ and $A$ are symmetric matrices. The relevance of this function comes from Saroglou's reformulation \cite{Saroglou14} of the log-Brunn--Minkowski conjecture. We show that $\Sigma$ and $A$ commute if and only if the inequality holds for every box of the form  $UD B_\infty^n$, where $U$ is orthogonal and $D$ is diagonal. In the non-commuting case, this yields counterexamples
for orthogonal boxes whose axes can be chosen arbitrarily close to the coordinate axes. We also show that the inequality holds when $K$ is unconditional with respect to an eigenbasis of $\Sigma$ and $A$ is an arbitrary symmetric matrix.
\end{abstract}


\section{Introduction}
We fix an inner product $\langle \cdot,\cdot\rangle$ on $\R^n$ and an orthonormal basis $e_1,\ldots,e_n$. By a convex body we mean a compact convex set with nonempty interior. Unless otherwise stated, unconditional convex bodies and diagonal matrices are understood with respect to this fixed basis. We denote by $B_p^n$ the $\ell_p$-unit ball with respect to this basis. By an orthogonal box, we mean a box whose sides are parallel to the vectors of some orthonormal basis, while a coordinate box is an orthogonal box whose sides are parallel to the fixed basis $e_1,\ldots,e_n$. We write $\mathcal{S}$ and $\mathcal{D}$ for the spaces of symmetric and diagonal matrices on $\R^n$, respectively, and by $\mathcal{S}_{\succ0}$ and $\mathcal{D}_{\succ0}$ their positive-definite elements. 

For $\Sigma\in\mathcal S_{\succ0}$, let $\gamma_\Sigma$ denote the centered Gaussian probability measure on $\R^n$ with covariance matrix $\Sigma$. For an origin-symmetric convex set $K\subseteq\R^n$, let $\mu_{K,\Sigma}$ denote the Gaussian measure $\gamma_\Sigma$ restricted to $K$, namely the even log-concave probability measure given by
\begin{equation*}
d\mu_{K,\Sigma}(y)=\frac{\mathbf{1}_K(y)}{\gamma_\Sigma(K)}\,d\gamma_\Sigma(y).
\end{equation*}

We study the following question.
\begin{question}\label{Kproblem}
For which triples $(K,\Sigma,A)$, where $K\subseteq\R^n$ is an origin-symmetric convex set, $\Sigma\in\mathcal{S}_{\succ0}$, and $A\in\mathcal{S}$, does the following inequality hold?
\begin{align}\label{VarlogBM1}
\operatorname{Var}_{\mu_{K,\Sigma}} \left(\left\langle A\Sigma^{-1}y,y\right\rangle\right)\leq\int_{\R^n}\left\langle(A^2\Sigma^{-1}+A\Sigma^{-1}A)y,y\right\rangle\,d\mu_{K,\Sigma}(y).
\end{align}
\end{question}
We note two special cases. When $K=\R^n$, inequality \eqref{VarlogBM1} holds with equality for every $\Sigma\in\mathcal{S}_{\succ0}$ and every $A\in\mathcal{S}$. This follows from the trace formulas for Gaussian quadratic forms (see Appendix A). Another simple case is when both $\Sigma$ and $A$ are diagonal and $K$ is a coordinate box. In this setting, the inequality reduces to a one-dimensional inequality. This case is contained in the more general commuting setting discussed below in Theorem~\ref{propdirectionalcommuting1}.

\vspace{0.5cm} 
Inequality \eqref{VarlogBM1} is closely related to the log-Brunn--Minkowski conjecture of B\"or\"oczky, Lutwak, Yang, and Zhang \cite{BLYZ12}. Recall that the support function of a convex body $K\subseteq\R^n$ is defined by $h_K(y)=\sup_{x\in K}\langle x,y\rangle$. The logarithmic mean of origin-symmetric convex bodies $K,L\subseteq\R^n$, for $\lambda\in[0,1]$, is given by
\begin{equation*}
(1-\lambda)K+_o\lambda L =\left\{ x\in\R^n: \langle x,y\rangle \leq h_K(y)^{1-\lambda}h_L(y)^\lambda \ \text{for all }y\in\R^n \right\}.
\end{equation*}
Since the geometric mean $h_K^{1-\lambda}h_L^\lambda$ need not be a support function, $(1-\lambda)K+_o\lambda L$ is defined as the largest convex body whose support function is bounded above by $h_K^{1-\lambda}h_L^\lambda$. B\"or\"oczky, Lutwak, Yang, and Zhang \cite{BLYZ12} conjectured the following.

\begin{conjecture}[Log-Brunn--Minkowski conjecture]
For origin-symmetric convex bodies $K,L\subseteq\R^n$ and $\lambda\in[0,1]$,
\begin{equation*}
\big|(1-\lambda)K+_o\lambda L\big| \geq |K|^{1-\lambda}|L|^\lambda.
\end{equation*}
\end{conjecture}

The logarithmic Brunn--Minkowski conjecture was proved in dimension $n=2$ by B\"or\"oczky, Lutwak, Yang, and Zhang \cite{BLYZ12}. In dimension $n=3$, the inequality was proved for two special classes of convex bodies \cite{YZ19}. In higher dimensions, it is known under additional symmetry assumptions, including for unconditional convex bodies \cite{Saroglou15}, complex convex bodies \cite{Rotem14}, and convex bodies invariant under $n$ independent linear reflections \cite{BK22}. Local results are known near the Euclidean ball \cite{CLM17} and when the reference body is a zonoid \cite{vanHandel23}. Recent preprints treat the case of $n-2$ reflection symmetries \cite{Lu26} and the local problem for bodies of revolution \cite{Iffland26}.

Combining Saroglou's result \cite{Saroglou14} with the second-variation computation of Cordero-Erausquin and Rotem \cite{CR20} yields the following equivalent formulation; see  section~\ref{PrositionEquivalence11} for the details.

\begin{prop}\label{PrositionEquivalence}
The following are equivalent:
\begin{enumerate}[(i)]
\item The log-Brunn--Minkowski conjecture holds.
\item Inequality \eqref{VarlogBM1} holds for $K=B^n_{\infty}$, any $\Sigma\in\mathcal{S}_{\succ0}$, and any $A\in\mathcal{D}$.
\end{enumerate}
\end{prop}

Second variation arguments lead to quadratic-form formulations of local $L_p$-Brunn--Minkowski inequalities and, equivalently, to spectral formulations in the smooth setting. Colesanti \cite{Colesanti08} derived Poincar\'e-type inequalities on the boundary of a convex body from the Brunn--Minkowski inequality. Colesanti, Livshyts, and Marsiglietti \cite{CLM17} formulated the local log-Brunn--Minkowski inequality as a Poincar\'e-type inequality for the cone-volume measure. Kolesnikov and Milman \cite{KoM22} related local $L_p$-Brunn--Minkowski inequalities to spectral-gap estimates for the Hilbert--Brunn--Minkowski operator. Further developments of the local theory can be found in \cite{CHLL20,Put21,Milman25,KLR25}. Related Brunn--Minkowski inequalities for Gaussian measure were established in \cite{KL21,EsM21}. In the present setting, the corresponding Poincar\'e-type inequality \eqref{VarlogBM1} is formulated for the Gaussian measure restricted to the convex body, whereas the local inequalities above are formulated in terms of measures on the sphere.

\vspace{0.5cm} 

We next consider the case where $\Sigma$ and $A$ commute, or equivalently, share a common eigenbasis. We will show that in this case \eqref{VarlogBM1} holds for every origin-symmetric convex set $K$. This follows from the symmetry-improved Gaussian Poincar\'e inequality of Cordero-Erausquin, Fradelizi, and Maurey \cite{CFM04}, after a linear change of variables reducing the covariance matrix $\Sigma$ to the identity. More precisely, using the approximation argument described just before
\cite[equation~(8)]{CFM04}, together with \cite[Remark~1, inequality~(14)]{CFM04}, we have
\begin{equation}\label{EvenGaussianPoincare}
\operatorname{Var}_{\mu_{K,I_n}}(f)\leq\frac{1}{2}\int_{\R^n}|\nabla f(x)|^2\,d\mu_{K,I_n}(x),
\end{equation}
for every origin-symmetric convex set $K\subseteq\R^n$ and every smooth even function $f:\R^n\to\R$ such that $f,\nabla f\in L^2(\mu_{K,I_n})$. This is the symmetry-improved Gaussian Poincar\'e inequality, with constant $1/2$ instead of $1$, which reflects the symmetry of $K$ and the evenness of $f$. The commuting case can in fact be characterized in terms of \eqref{VarlogBM1}, as shown in the following theorem.

Our result in the commuting setting is the following.

\begin{theorem}\label{propdirectionalcommuting1}
Let $\Sigma\in\mathcal S_{\succ0}$ and
$A\in\mathcal D$. Then the following are equivalent:
\begin{enumerate}[(i)]
\item $A$ and $\Sigma$ commute.
\item Inequality \eqref{VarlogBM1} holds for every origin-symmetric orthogonal box $K\subseteq\R^n$.
\end{enumerate}
\end{theorem}

The proof of Theorem~\ref{propdirectionalcommuting1} gives a simple way to construct counterexamples to \eqref{VarlogBM1} in the non-commuting case. Namely, given $A$ and $\Sigma$ which do not commute, one can always find an origin-symmetric orthogonal box $K$ for which \eqref{VarlogBM1} fails. The boxes obtained by this construction are not coordinate boxes; a counterexample with a coordinate box would give a negative answer to the log-Brunn--Minkowski conjecture. Nevertheless, the construction yields orthogonal boxes arbitrarily close to coordinate boxes. Cordero-Erausquin and Rotem \cite[Theorem~6]{CR20} used a thin-box argument and a second-variation computation to show that the strong $(B)$-property fails for Gaussian measures with covariance matrices arbitrarily close to the identity. In the present setting, this idea leads from \eqref{VarlogBM1} to a directional inequality, and requiring this inequality in every direction is equivalent to the commutativity of $A$ and $\Sigma$.

\vspace{0.5cm} 

Theorem~\ref{propdirectionalcommuting1} shows that commutativity is the natural condition guaranteeing \eqref{VarlogBM1} over all origin-symmetric orthogonal boxes. In the non-commuting case, the geometry of $K$ becomes important. Although the inequality may fail for some orthogonal boxes, additional symmetries of the body with respect to an eigenbasis of $\Sigma$ can still guarantee its validity. The next theorem shows that unconditionality is sufficient when $\Sigma$ is diagonal. 

Our result in the non-commuting setting is the following.
\begin{theorem}\label{unconditionalK}
Inequality \eqref{VarlogBM1} holds for any unconditional convex body $K$, any $\Sigma\in\mathcal{D}_{\succ0}$, and any $A\in\mathcal{S}$.
\end{theorem}

As a consequence of Theorem~\ref{unconditionalK}, \eqref{VarlogBM1} holds for $K=B_2^n$ and arbitrary $\Sigma\in\mathcal{S}_{\succ0}$, since the Euclidean ball is unconditional with respect to every orthonormal basis.

\vspace{0.5cm}

The remainder of the paper is devoted to the proofs of these results.

\section{Proof of Proposition~\ref{PrositionEquivalence}}\label{PrositionEquivalence11}

Saroglou \cite{Saroglou14} (see also \cite{CR20}) proved that the log-Brunn--Minkowski conjecture is equivalent to the following statement: for every symmetric positive-definite matrix $\Sigma\in\mathcal{S}_{\succ0}$ and every diagonal matrix $A\in\mathcal{D}$, the function
\begin{equation}\label{mainFun}
t\longmapsto\gamma_{\Sigma}\left(e^{tA}B_\infty^n\right)
\end{equation}
is log-concave. We briefly explain this equivalence. By \cite[Theorem~3.1, p.~6]{Saroglou14}, the log-Brunn--Minkowski conjecture implies the corresponding inequality for Gaussian measure. Hence,
\cite[Theorem~5.1(iii), p.~15]{Saroglou14} implies that
\begin{equation*}
\R\ni t\longmapsto \int_{e^{tA}B_\infty^n}\varphi_n(Tx)\,dx
\end{equation*}
is log-concave for every $T\in GL_n$ and every $A\in\mathcal D$, where $\varphi_n$ denotes the density of the standard Gaussian measure $\gamma_n$. Using the identity
\begin{equation}\label{GaussianLinearImageIdentity}
\int_E\varphi_n(Tx)\,dx = \frac{1}{|\det T|}\gamma_{(T^T T)^{-1}}(E),
\end{equation}
we obtain the log-concavity of \eqref{mainFun} for every $\Sigma\in\mathcal S_{\succ0}$, since every such $\Sigma$ can be written as $\Sigma=(T^T T)^{-1}$ for some $T\in GL_n$. Conversely, the log-concavity of \eqref{mainFun}, together with \eqref{GaussianLinearImageIdentity}, gives the hypothesis of \cite[Theorem~5.1(i), p.~14]{Saroglou14}, and therefore yields the log-Brunn--Minkowski conjecture in every dimension.

\vspace{0.5cm} 

We will show that taking the second variation of \eqref{mainFun} leads to the variance inequality \eqref{VarlogBM1}. For diagonal $A$, \eqref{VarlogBM1} is the second-variation condition associated with the strong $(B)$-property studied in \cite{CR20}; see also \cite{CR23,Put26}.

\begin{lemma}\label{ComputationSecondDerivative}
Let $K\subseteq\R^n$, $\Sigma\in\mathcal{S}_{\succ0}$, and $A\in\mathcal{S}$. Consider the function
\begin{equation*}
g(t) = \int_{K}e^{-\frac12\left\langle e^{tA}\Sigma^{-1}e^{tA}y,y	\right\rangle} \,dy.
\end{equation*}
Then
\begin{equation}
(\log g)''(0) = \operatorname{Var}_{\mu_{K,\Sigma}} \left( \left\langle A\Sigma^{-1}y,y\right\rangle \right)-\int_{\R^n}\left\langle(A^2\Sigma^{-1}+A\Sigma^{-1}A)y,y\right\rangle\,d\mu_{K,\Sigma}(y).
\end{equation}
\end{lemma}

\begin{proof}
Differentiating $g$, we obtain
\begin{align*}
g'(t)
&=-\frac{1}{2}\int_{K}\left\langle\frac{d}{dt}\left(e^{tA}\Sigma^{-1}e^{tA}\right)y,y\right\rangle e^{-\frac12 \left\langle e^{tA}\Sigma^{-1}e^{tA}y,y\right\rangle} \,dy\\
&=-\frac{1}{2}\int_{K}\left\langle\left(Ae^{tA}\Sigma^{-1}e^{tA}+e^{tA}\Sigma^{-1}e^{tA}A\right)y,y\right\rangle e^{-\frac{1}{2}\left\langle e^{tA}\Sigma^{-1}e^{tA}y,y \right\rangle} \,dy\\
&=-\int_{K} \left\langle Ae^{tA}\Sigma^{-1}e^{tA}y,y \right\rangle e^{-\frac12 \left\langle e^{tA}\Sigma^{-1}e^{tA}y,y \right\rangle} \,dy.
\end{align*}
For the last equality, we used the fact that both $A$ and $e^{tA}\Sigma^{-1}e^{tA}$ are symmetric.

Differentiating once more, we obtain
\begin{align*}
g''(t)
&=\int_{K}\Big[-\left\langle\left(A^2e^{tA}\Sigma^{-1}e^{tA}+Ae^{tA}\Sigma^{-1}e^{tA}A \right)y,y \right\rangle\\
&\hspace{3cm}+\left\langle Ae^{tA}\Sigma^{-1}e^{tA}y,y \right\rangle^2 \Big] e^{-\frac12 \left\langle e^{tA}\Sigma^{-1}e^{tA}y,y \right\rangle }\,dy.
\end{align*}
In particular,
\begin{align*}
g'(0)=-\int_{K}\left\langle A\Sigma^{-1}y,y\right\rangle e^{-\frac12\langle\Sigma^{-1}y,y\rangle} \,dy,
\end{align*}
and
\begin{align*}
g''(0)=\int_{K}\Big[-\left\langle A^2\Sigma^{-1}y,y\right\rangle-\left\langle A\Sigma^{-1}Ay,y\right\rangle+\left\langle A\Sigma^{-1}y,y\right\rangle^2\Big]e^{-\frac12\langle\Sigma^{-1}y,y\rangle}\,dy.
\end{align*}
Since $	(\log g)''(0)=g''(0)/g(0)-g'(0)^2/g(0)^2$, the Lemma follows from the definition of $\mu_{K,\Sigma}$.
\end{proof}

\begin{proof}[Proof of Proposition~\ref{PrositionEquivalence}]
By Saroglou's result \cite{Saroglou14}, the log-Brunn--Minkowski conjecture is equivalent to the log-concavity of \eqref{mainFun} for every $\Sigma\in\mathcal{S}_{\succ0}$ and every $A\in\mathcal{D}$. We show that this is equivalent to \eqref{VarlogBM1} for $K=B_\infty^n$, every $\Sigma\in\mathcal{S}_{\succ0}$, and every $A\in\mathcal{D}$.

Let $F(t)=\gamma_{\Sigma}\left(e^{tA}B_\infty^n\right)$. The log-concavity of the function in \eqref{mainFun} is equivalent to $(\log F)''(t)\leq 0$ for every $t\in\R$. We claim that it is enough to verify this inequality at $t=0$ for every $\Sigma\in\mathcal{S}_{\succ0}$. Indeed, fix $t_0\in\R$ and set
\begin{equation*}
\Sigma_{t_0}=e^{-t_0A}\Sigma e^{-t_0A}.
\end{equation*}
Since $\Sigma$ is symmetric and positive definite, so is $\Sigma_{t_0}$. If $X\sim N(0,\Sigma)$, then
\begin{equation*}
e^{-t_0A}X\sim N(0,\Sigma_{t_0}).
\end{equation*}
Consequently, for every $s\in\R$,
\begin{align*}
\gamma_{\Sigma}\left(e^{(t_0+s)A}B_\infty^n\right)
&=\mathbb{P}\left(X\in e^{(t_0+s)A}B_\infty^n\right)\\
&=\mathbb{P}\left(e^{-t_0A}X\in e^{sA}B_\infty^n\right)\\
&=\gamma_{\Sigma_{t_0}}\left(e^{sA}B_\infty^n\right).
\end{align*}
Thus, replacing $t$ by $t_0+s$ corresponds to replacing $\Sigma$ by $\Sigma_{t_0}$. Therefore, the log-concavity condition at $t_0$ for $\Sigma$ is the same as the log-concavity condition at $0$ for
$\Sigma_{t_0}$. Since $t_0$ is arbitrary and $\Sigma_{t_0}\in\mathcal{S}_{\succ0}$, it is enough to verify $(\log F)''(0)\leq 0$ for every $\Sigma\in\mathcal{S}_{\succ0}$ and every $A\in\mathcal{D}$.

Applying the change of variables $x=e^{tA}y$, we obtain
\begin{align*}
F(t)=\gamma_{\Sigma}(e^{tA}B_\infty^n)
&=\frac{1}{c_{\Sigma}}\int_{e^{tA}B_\infty^n}e^{-\frac12\langle\Sigma^{-1}x,x\rangle}\,dx\\
&=\frac{e^{t\operatorname{tr}(A)}}{c_{\Sigma}}\int_{B_\infty^n}e^{-\frac12	\left\langle e^{tA}\Sigma^{-1}e^{tA}y,y\right\rangle}\,dy\\
&=\frac{e^{t\operatorname{tr}(A)}}{c_{\Sigma}}g(t),
\end{align*}
where $g$ is the function in Lemma~\ref{ComputationSecondDerivative} with $K=B_\infty^n$ and $c_{\Sigma}$ is the normalizing constant. Since
\begin{equation*}
\log F(t)=t\operatorname{tr}(A)-\log c_{\Sigma}+\log g(t),
\end{equation*}
we have $(\log F)''(0)=(\log g)''(0)$. By Lemma~\ref{ComputationSecondDerivative}, the condition $(\log F)''(0)\leq 0$ is therefore equivalent to
\begin{align*}
\operatorname{Var}_{\mu_{B_\infty^n,\Sigma}}\left(\left\langle A\Sigma^{-1}y,y\right\rangle\right)\leq\int_{\R^n}\left\langle(A^2\Sigma^{-1}+A\Sigma^{-1}A)y,y\right\rangle\,d\mu_{B_\infty^n,\Sigma}(y),
\end{align*}
which is \eqref{VarlogBM1} for $K=B_\infty^n$. This proves the equivalence.
\end{proof}

\section{The Commuting Case}\label{SigmaEvenPoincare}

We now prove Theorem~\ref{propdirectionalcommuting1}. We first extend the symmetry-improved Gaussian Poincar\'e inequality \eqref{EvenGaussianPoincare} to Gaussian measures with
arbitrary covariance matrices.

\begin{lemma}
Let $K\subseteq\R^n$ be an origin-symmetric convex set and let $\Sigma\in\mathcal S_{\succ0}$. Then, for every smooth even function $f:\R^n\to\R$ such that $f,\nabla f\in L^2(\mu_{K,\Sigma})$, it holds
\begin{equation}\label{EvenGaussianPoincareSigma1}
\operatorname{Var}_{\mu_{K,\Sigma}}(f) \leq \frac{1}{2} \int_{\R^n} \langle \Sigma\nabla f(x),\nabla f(x)\rangle \,d\mu_{K,\Sigma}(x).
\end{equation}
\end{lemma}

\begin{proof}
Set $\widetilde K:=\Sigma^{-1/2}K$. Since $K$ is origin-symmetric and convex, so is $\widetilde K$. Making the change of variables $x=\Sigma^{1/2}z$, we have $\langle \Sigma^{-1}x,x\rangle=|z|^2$ and $dx=(\det\Sigma)^{1/2}\,dz$. Hence $\gamma_{\Sigma}(K)=\gamma_n(\widetilde K)$, and, for every integrable function $h:\R^n\to\R$,
\begin{equation}\label{changevariablesrestgaussian}
\int_{\R^n}h(x)\,d\mu_{K,\Sigma}(x) = \int_{\R^n} h(\Sigma^{1/2}z)\,d\mu_{\widetilde K,I_n}(z).
\end{equation}
Next, set $g(z):=f(\Sigma^{1/2}z)$. Since $f$ is even, so is $g$. By \eqref{changevariablesrestgaussian} and \eqref{EvenGaussianPoincare} we have
\begin{align*}
\operatorname{Var}_{\mu_{K,\Sigma}}(f)
&=\operatorname{Var}_{\mu_{\widetilde K,I_n}}(g)\\
&\leq\frac12 \int_{\R^n} |\nabla g(z)|^2 \,d\mu_{\widetilde K,I_n}(z)\\
&=\frac12 \int_{\R^n} \left\langle \Sigma\nabla f(\Sigma^{1/2}z), \nabla f(\Sigma^{1/2}z) \right\rangle \,d\mu_{\widetilde K,I_n}(z)\\
&=\frac12 \int_{\R^n} \left\langle \Sigma\nabla f(x), \nabla f(x) \right\rangle \,d\mu_{K,\Sigma}(x),
\end{align*}
where $\nabla g(z)=\Sigma^{1/2}\nabla f(\Sigma^{1/2}z)$. This proves the lemma.
\end{proof}

We can now obtain \eqref{VarlogBM1} in the commuting case.

\begin{lemma}\label{CommutativeCase}
Let $\Sigma\in\mathcal S_{\succ0}$ and $A\in\mathcal S$. If $\Sigma$ and $A$ commute, then \eqref{VarlogBM1} holds for every origin-symmetric convex set $K\subseteq\R^n$.
\end{lemma}

\begin{proof}
We apply \eqref{EvenGaussianPoincareSigma1} to the even function $f(y)=\left\langle A\Sigma^{-1}y,y\right\rangle$. Since $A$ and $\Sigma$ commute, $A\Sigma^{-1}$ is symmetric, and hence $\nabla f(y)=2A\Sigma^{-1}y$. Therefore,
\begin{align*}
\frac12 \left\langle \Sigma\nabla f(y),\nabla f(y) \right\rangle
&=2\left\langle\Sigma A\Sigma^{-1}y,A\Sigma^{-1}y\right\rangle\\
&=2\left\langle A^2\Sigma^{-1}y,y \right\rangle\\
&=\left\langle \left( A^2\Sigma^{-1}+A\Sigma^{-1}A \right)y,y \right\rangle,
\end{align*}
where in the last equality we again used that $A$ and $\Sigma$ commute. Since $f$ is quadratic, $f,\nabla f\in L^2(\mu_{K,\Sigma})$, and thus \eqref{EvenGaussianPoincareSigma1} yields \eqref{VarlogBM1}.
\end{proof}

The quantities $\alpha,\beta,\gamma$ below arise from the one-dimensional thin-box limit of \eqref{VarlogBM1}. This leads to the following characterization of the commutativity of $A$ and $\Sigma$.

\begin{theorem}\label{directionalInequality}
Let $\Sigma\in\mathcal S_{\succ0}$ and $A\in\mathcal D$. For $u\in\R^n$ set
\begin{align*}
\alpha(u) &:=\left\langle \Sigma^{-1}u,u\right\rangle,\\
\beta(u)  &:=\left\langle A\Sigma^{-1}u,u\right\rangle,\\
\gamma(u) &:=\left\langle \left( A^2\Sigma^{-1}+A\Sigma^{-1}A \right)u,u \right\rangle.
\end{align*}
Then the following are equivalent:
\begin{enumerate}[(i)]
\item $A$ and $\Sigma$ commute.
\item $2\beta(u)^2\leq\alpha(u)\gamma(u)$ for every $u\in\R^n$.
\item Inequality \eqref{VarlogBM1} holds for every origin-symmetric orthogonal box $K\subseteq\R^n$.
\end{enumerate}
\end{theorem}

\begin{proof} We prove the implications $(i)\Rightarrow(iii)\Rightarrow(ii)\Rightarrow(i)$.

$(i)\implies (iii)$ Follows from Lemma \ref{CommutativeCase}.

$(iii)\implies(ii)$
Fix $u\in S^{n-1}$. Since $\Sigma\in\mathcal S_{\succ0}$, we have $\alpha(u)>0$. Complete $u$ to an orthonormal basis $u,u_2,\ldots,u_n$. For $\varepsilon>0$, let
\begin{equation}\label{thinBOX}
K_\varepsilon =\left\{t u+\sum_{j=2}^n s_j u_j: |t|\leq\frac1\varepsilon,\quad |s_j|\leq\varepsilon \right\}.
\end{equation}
Thus, in the basis $u,u_2,\ldots,u_n$, the set $K_\varepsilon$ is the box $\left[-\frac{1}{\varepsilon},\frac{1}{\varepsilon}\right] \times[-\varepsilon,\varepsilon]^{n-1}$. For every $h\in C^1(\R^n)$ such that $h$ and $\nabla h$ have at most exponential growth, we have
\begin{equation}\label{Diverge}
\lim_{\varepsilon\to0} \int_{\R^n}h(x)\,d\mu_{K_\varepsilon,\Sigma}(x) =\int_{\R}h(tu)\,d\nu_{u}(t),
\end{equation}
where
\begin{equation}
\nu_{u}
=N\left(0,\frac1{\langle\Sigma^{-1}u,u\rangle}\right)
=N\left(0,\frac1{\alpha(u)}\right),
\end{equation}
(see Appendix~B for the precise statement and proof). Next, apply \eqref{VarlogBM1} to $K_\varepsilon$ and let $\varepsilon\to0$. Since $\beta(tu)=t^2\beta(u)$ and $\gamma(tu)=t^2\gamma(u)$, using \eqref{Diverge} for $\beta$, $\beta^2$, and $\gamma$, we obtain
\begin{equation}\label{1DIMlimiting}
\operatorname{Var}_{\nu_u}\bigl(\beta(u)t^2\bigr)
\leq
\int_{\R}\gamma(u)t^2\,d\nu_u(t).
\end{equation}
Since $\nu_u=N(0,\alpha(u)^{-1})$, we have
\begin{align*}
\int_{\R}t^2\,d\nu_u(t)&=\frac{1}{\alpha(u)},\\
\int_{\R}t^4\,d\nu_u(t)&=\frac{3}{\alpha(u)^2},
\end{align*}
Hence \eqref{1DIMlimiting} becomes $\frac{2\beta(u)^2}{\alpha(u)^2} \leq \frac{\gamma(u)}{\alpha(u)}$, which gives $(ii)$.

$(ii)\implies(i)$ Assume that $2\beta(u)^2\leq\alpha(u)\gamma(u)$
for every $u\in\R^n$, and write $\Sigma^{-1}=(s_{ij})_{i,j=1}^n,$ and $A=\operatorname{diag}(a_1,\ldots,a_n)$. To prove that $A$ and $\Sigma$ commute, it is enough to show that
\begin{equation*}
A\Sigma^{-1}-\Sigma^{-1}A
=\bigl((a_i-a_j)s_{ij}\bigr)_{i,j=1}^n
\end{equation*}
is the zero matrix. Its diagonal entries are automatically zero, and the $(i,j)$ entry is also zero whenever $a_i=a_j$. It therefore remains to consider indices $i\neq j$ such that $a_i\neq a_j$.

Fix such a pair $i,j$. Define the polynomial function
\begin{equation*}
F(t)=\alpha(e_i+t e_j)\gamma(e_i+t e_j)-2\beta(e_i+t e_j)^2.
\end{equation*}
By assumption, $F(t)\geq0$ for every $t\in\R$. Moreover, one can check that $F(0)=0$. Thus $t=0$ is a minimum of $F$, so $F'(0)=0$. Notice that
\begin{align*}
\alpha(e_i+te_j)
&=s_{ii}+2t s_{ij}+t^2s_{jj},\\
\beta(e_i+te_j)
&=a_i s_{ii}+t(a_i+a_j)s_{ij}+t^2a_j s_{jj},\\
\gamma(e_i+te_j)
&=2a_i^2s_{ii}+t(a_i+a_j)^2s_{ij}+2t^2a_j^2s_{jj}.
\end{align*}
Therefore,
\begin{align*}
F'(0)
&=\left.\frac{d}{dt}\alpha(e_i+te_j)\right|_{t=0}\gamma(e_i)+\alpha(e_i)\left.\frac{d}{dt}\gamma(e_i+te_j)\right|_{t=0} -4\beta(e_i)\left.\frac{d}{dt}\beta(e_i+te_j)\right|_{t=0}\\
&=2s_{ij}\cdot 2a_i^2s_{ii} +s_{ii}(a_i+a_j)^2s_{ij}-4a_is_{ii}(a_i+a_j)s_{ij}\\
&=s_{ii}s_{ij}\left(4a_i^2+(a_i+a_j)^2-4a_i(a_i+a_j)\right)\\
&=s_{ii}s_{ij}(a_i-a_j)^2.
\end{align*}
Since $a_i\neq a_j$ and $\Sigma^{-1}\succ0$ implies $s_{ii}>0$, it follows that $s_{ij}=0$. Thus $(a_i-a_j)s_{ij}=0$ for every $i,j$, and this proves $(i)$.
\end{proof}

\begin{remark}
We also give a direct proof of $(i)\Rightarrow(ii)$. Suppose that $A$ and $\Sigma$ commute. Since both matrices are symmetric, they are simultaneously diagonalizable in an orthonormal basis. Thus we may write $A=\operatorname{diag}(a_1,\ldots,a_n)$ and $\Sigma^{-1}=\operatorname{diag}(p_1,\ldots,p_n)$ with $p_i>0$. Then, for $u=(u_1,\ldots,u_n)$, the Cauchy--Schwarz inequality gives
\begin{align*}
2\beta(u)^2=2\left(\sum_{i=1}^n a_i p_i u_i^2\right)^2 \leq 2\left(\sum_{i=1}^n p_i u_i^2\right)\left(\sum_{i=1}^n a_i^2p_i u_i^2\right)=\alpha(u)\gamma(u).
\end{align*}
\end{remark}

\begin{remark}
Let $\Sigma\in\mathcal S_{\succ0}$ and $A\in\mathcal D$. Then the directional inequality is always an equality in each coordinate direction $e_i$. Indeed, writing $A=\operatorname{diag}(a_1,\ldots,a_n)$ and $\Sigma^{-1}=(s_{ij})_{i,j=1}^n,$ we have
\begin{equation*}
2\beta(e_i)^2=2a_i^2s_{ii}^2=\alpha(e_i)\gamma(e_i).
\end{equation*}
\end{remark}

\begin{remark}
The proof of Theorem~\ref{directionalInequality} also gives a simple way to construct counterexamples in the non-commuting case. If $A$ and $\Sigma$ do not commute, then the one-dimensional limiting inequality fails for some $u\in S^{n-1}$. Hence, for any sufficiently small $\varepsilon>0$, the thin box $K_\varepsilon$ defined in 
\eqref{thinBOX} yields a counterexample to \eqref{VarlogBM1}.
\end{remark}

\begin{remark}\label{closeTOaxis}
In the following two-dimensional example, arbitrarily small rotations of sufficiently thin rectangles yield counterexamples. Consider
\[
A=
\begin{pmatrix}
1&0\\
0&-1
\end{pmatrix},
\qquad \Sigma^{-1} =
\begin{pmatrix}
1&-1/2\\
-1/2&1
\end{pmatrix}.
\qquad
\]
For $u_\theta=(\cos\theta,\sin\theta)$, a direct computation gives
\begin{align*}
\alpha(u_\theta) &=1-\frac12\sin(2\theta),\\
\beta(u_\theta)  &=\cos(2\theta),\\
\gamma(u_\theta) &=2.
\end{align*}
Thus the directional inequality $2\beta(u_\theta)^2 \leq \alpha(u_\theta)\gamma(u_\theta)$ is equivalent to
\begin{equation*}
\sin(2\theta) \left(\sin(2\theta)-\frac12\right) \geq0.
\end{equation*}
Hence it fails whenever $0<\sin(2\theta)<\frac12$. In particular, such directions $u_\theta$ can be chosen arbitrarily close to $e_1$ or $e_2$, and therefore the corresponding counterexample boxes can be chosen with their axes arbitrarily close to the coordinate axes.
\end{remark}

\section{On the non-commutative case}\label{KuncondtoSIGMA}

In this section we prove Theorem \ref{unconditionalK} which we restate here.

\begin{theoremrep}
For any unconditional convex body $K$, any $\Sigma\in\mathcal{D}_{\succ0}$, and any $A\in\mathcal{S}$, we have
\begin{align}\label{VarlogBM11}
\operatorname{Var}_{\mu_{K,\Sigma}} \left( \left\langle A\Sigma^{-1}y,y\right\rangle \right) \leq \int_{\R^n} \left\langle (A^2\Sigma^{-1}+A\Sigma^{-1}A)y,y \right\rangle\,d\mu_{K,\Sigma}(y).
\end{align}
\end{theoremrep}

Notice that the unconditionality of $K$ does not, in general, imply that the measure $\mu_{K,\Sigma}$ is unconditional for arbitrary $\Sigma\in\mathcal{S}_{\succ0}$. However, if in addition $\Sigma\in\mathcal{D}_{\succ0}$, then $\mu_{K,\Sigma}$ is unconditional.

\vspace{0.5cm}
The proof of Theorem \ref{unconditionalK} is based on decomposing $\left\langle A\Sigma^{-1}y,y\right\rangle$ into its diagonal and off-diagonal parts. The diagonal part is controlled by the even $\Sigma$-symmetry-improved Gaussian Poincar\'e inequality \eqref{EvenGaussianPoincareSigma1}, while for the off-diagonal part, the unconditionality of $\mu_{K,\Sigma}$ eliminates the mixed terms, and an application of the divergence theorem provides the desired bound for the remaining terms.

\begin{lemma}\label{unconditionalcancellations}
Let $\Sigma\in\mathcal{D}_{\succ0}$ and let	$K\subseteq\R^n$ be an unconditional convex body. Consider the probability measure
\begin{equation*}
d\mu(y) =\frac{1_K(y)}{\gamma_\Sigma(K)} \,d\gamma_\Sigma(y).
\end{equation*}
Then every monomial $q$ in which at least one coordinate appears with an odd exponent satisfies $\int q\,d\mu=0$. In particular,
\begin{enumerate}[(i)]
\item $\displaystyle \int_{\R^n} y_i y_j\,d\mu(y)=0$ for every $i\neq j$;
\item $\displaystyle \int_{\R^n} y_r^2y_i y_j\,d\mu(y)=0$ for every $r=1,\ldots,n$ and every $i\neq j$;
\item $\displaystyle \int_{\R^n} y_i y_j y_k y_\ell\,d\mu(y)=0$ whenever $i<j$, $k<\ell$, and $\{i,j\}\neq\{k,\ell\}$.
\end{enumerate}
\end{lemma}

\begin{proof}
Since $\Sigma\in\mathcal D_{\succ0}$, the Gaussian measure $\gamma_\Sigma$ is invariant under reflection in each coordinate. Since $K$ is unconditional, the same is true for the restricted measure $\mu$. Let $R_m$ denote reflection in the $m$-th coordinate, $m=1,\ldots,n$. Then $\mu$ is invariant under $R_m$. Therefore, for every integrable function $f$, if $f(R_m y)=-f(y)$ for some $m$, then \begin{align}\label{eq:reflection-principle}
\int_{\R^n}f\,d\mu =\int_{\R^n}f\circ R_m\,d\mu =-\int_{\R^n}f\,d\mu,
\end{align}
and hence
\begin{equation*}
\int_{\R^n}f\,d\mu=0.
\end{equation*}
Any monomial in which some coordinate appears with an odd exponent changes sign under reflection in that coordinate. Hence {\rm (i)--(iii)} follow.
\end{proof}

\begin{remark} Recall that, for $f_1,\ldots,f_m\in L^2(\mu)$, \begin{align}\label{eq:variance-sum} \operatorname{Var}_{\mu} \left( \sum_{\alpha=1}^m f_\alpha \right) &= \sum_{\alpha=1}^m \operatorname{Var}_{\mu}(f_\alpha) + 2\sum_{1\leq\alpha<\beta\leq m} \operatorname{Cov}_{\mu}(f_\alpha,f_\beta). \end{align} \end{remark}

\begin{lemma}\label{lemvariancedecomposition}
Let $\Sigma\in\mathcal{D}_{\succ0}$ and $A\in\mathcal{S}$, and let
$\mu$ be a probability measure on $\R^n$ satisfying properties {\rm (i)--(iii)} of Lemma~\ref{unconditionalcancellations}. Write $A=(a_{ij})_{i,j=1}^n$ and $\Sigma^{-1}=\operatorname{diag}(p_1,\ldots,p_n)$, where $p_i>0$. Then
\begin{align*}
\operatorname{Var}_{\mu}\left(\left\langle A\Sigma^{-1}y,y\right\rangle\right) =\operatorname{Var}_{\mu}\left(\sum_{i=1}^n a_{ii}p_i y_i^2\right)+\sum_{1\leq i<j\leq n}a_{ij}^2(p_i+p_j)^2\int_{\R^n}y_i^2y_j^2\,d\mu(y).
\end{align*}
\end{lemma}

\begin{proof}
By the symmetry of $A$,
\begin{align*}
\left\langle A\Sigma^{-1}y,y\right\rangle
&=\sum_{i,j=1}^n a_{ij}p_j y_i y_j\\
&=\sum_{i=1}^n a_{ii}p_i y_i^2 +\sum_{1\leq i<j\leq n} a_{ij}(p_i+p_j)y_i y_j\\
&=:D+\sum_{i<j}c_{ij}y_i y_j,
\end{align*}
where $D(y):=\sum_{i=1}^n a_{ii}p_i y_i^2$ and $c_{ij}:=a_{ij}(p_i+p_j)$. Hence, applying 
\eqref{eq:variance-sum},
\begin{align*}
\operatorname{Var}_{\mu} \left( \left\langle A\Sigma^{-1}y,y\right\rangle \right) = \operatorname{Var}_{\mu}(D)
&+\sum_{i<j} \operatorname{Var}_{\mu}(c_{ij}y_i y_j) \nonumber\\
&+2\sum_{i<j} c_{ij}\operatorname{Cov}_{\mu}(D,y_i y_j) \nonumber\\
&+2\sum_{\substack{i<j,\;k<\ell\\ \{i,j\}\neq\{k,\ell\}}}c_{ij}c_{k\ell} \operatorname{Cov}_{\mu}(y_i y_j,y_k y_\ell).
\label{decompositoin}
\end{align*}
By assumptions {\rm (i)} and {\rm (ii)}, $\operatorname{Cov}_{\mu}(D,y_i y_j)=0$ for $i<j$, and by assumptions {\rm (i)} and {\rm (iii)}, $\operatorname{Cov}_{\mu}(y_i y_j,y_k y_\ell)=0$ for $i<j$ and $k<\ell$, whenever $\{i,j\}\neq\{k,\ell\}$. Finally, using once more assumption {\rm (i)}, we obtain 
\begin{align*}
\operatorname{Var}_{\mu}(c_{ij}y_i y_j)=c_{ij}^2 \operatorname{Var}_{\mu}(y_i y_j)=a_{ij}^2(p_i+p_j)^2 \int_{\R^n} y_i^2y_j^2\,d\mu(y),
\end{align*}
and this finishes the proof.
\end{proof}

\begin{proof}[Proof of Theorem~\ref{unconditionalK}]

Write $\Sigma^{-1}=\operatorname{diag}(p_1,\ldots,p_n)$, where $p_i>0$ and $A=(a_{ij})_{i,j=1}^n$. Notice that $\mu$ is unconditional. Therefore, by assumption (i) of Lemma \ref{unconditionalcancellations}, for every matrix
$M=(M_{ij})_{i,j=1}^n$,
\begin{equation*}
\int_{\R^n} \langle My,y\rangle\,d\mu(y) =\sum_{i=1}^n M_{ii} \int_{\R^n} y_i^2\,d\mu(y).
\end{equation*}
Set $m_i:=\int_{\R^n}y_i^2\,d\mu(y)$. Hence, using the symmetry of $A$, the right-hand side of \eqref{VarlogBM1} is written
\begin{align*}
\int_{\R^n} \Big[ \left\langle A^2\Sigma^{-1}y,y\right\rangle +\left\langle A\Sigma^{-1}Ay,y\right\rangle \Big] \,d\mu(y) 
&=\sum_{i=1}^n \left( A^2\Sigma^{-1}+A\Sigma^{-1}A \right)_{ii}m_i\\
&=\sum_{i=1}^n\sum_{j=1}^n a_{ij}^2(p_i+p_j)m_i\\
&=2\sum_{i=1}^n a_{ii}^2p_i m_i +\sum_{1\leq i<j\leq n} a_{ij}^2(p_i+p_j)(m_i+m_j). 
\end{align*}
We next consider the left-hand side of
\eqref{VarlogBM1}.
Since $\mu$ is unconditional, assumptions (i)--(iii) of Lemma~\ref{unconditionalcancellations} are satisfied and therefore, Lemma~\ref{lemvariancedecomposition} gives 
\begin{align}\label{eqvardecomposunco}
\operatorname{Var}_{\mu} \left( \left\langle A\Sigma^{-1}y,y\right\rangle\right)=\operatorname{Var}_{\mu}(q_0)+\sum_{1\leq i<j\leq n} a_{ij}^2(p_i+p_j)^2 \int_{\R^n} y_i^2y_j^2\,d\mu(y).
\end{align}
where $q_0(y):=\sum_{i=1}^n a_{ii}p_i y_i^2$. We claim that
\begin{align}\label{csdcrevevre}
\operatorname{Var}_{\mu}(q_0) \leq 2\sum_{i=1}^n a_{ii}^2p_i m_i.
\end{align}
Indeed, the function $q_0(y)=\sum_{i=1}^n a_{ii}p_i y_i^2$ is even, so applying the $\Sigma$-symmetry-improved Gaussian Poincar\'e inequality \eqref{EvenGaussianPoincareSigma1} with $\Sigma=\operatorname{diag}(p_1^{-1},\ldots,p_n^{-1})$, we obtain
\begin{align*}
\operatorname{Var}_{\mu}(q_0)
&\leq \frac12 \int_{\R^n} \langle \Sigma\nabla q_0(y),\nabla q_0(y)\rangle \,d\mu(y)\\
&=\frac12 \int_{\R^n} 4\sum_{i=1}^n \frac{1}{p_i} a_{ii}^2p_i^2y_i^2 \,d\mu(y)\\
&=2\sum_{i=1}^n a_{ii}^2p_i m_i.
\end{align*}
Next, fix $i\neq j$ and set $w(y):=e^{-\frac12\sum_{k=1}^n p_k y_k^2}$. Then
\begin{align*}
\frac{\partial}{\partial y_i} \left( y_i y_j^2w(y) \right)
&= y_j^2w(y) -p_i y_i^2y_j^2w(y)\\
&= \left( y_j^2-p_i y_i^2y_j^2 \right)w(y).
\end{align*}
Since $K$ is a convex body, the outer unit normal $\nu_K$ exists $\mathcal H^{n-1}$-almost everywhere on $\partial K$, and the Gauss--Green divergence theorem applies. Thus, for the vector field $F(y):=y_i y_j^2 w(y)e_i$, we obtain
\begin{align}
\int_K \left(y_j^2-p_i y_i^2y_j^2\right)w(y)\,dy
&=\int_K \operatorname{div}F(y)\,dy \notag\\
&=\int_{\partial K} \langle F(y),\nu_K(y)\rangle \,d\mathcal H^{n-1}(y) \notag\\
&=\int_{\partial K} y_i y_j^2 w(y) \langle e_i,\nu_K(y)\rangle \,d\mathcal H^{n-1}(y) 
\label{ithdenkserw}
\end{align}
where $\nu_K(y)$ denotes the outer unit normal to $\partial K$. By convexity, $\langle z-y,\nu_K(y)\rangle\leq0$ for every $z\in K$ and for $\mathcal H^{n-1}$-almost every $y\in\partial K$. Since $K$ is unconditional, $y-y_i e_i\in K$. Taking $z=y-y_i e_i$, we obtain $\langle -y_i e_i,\nu_K(y)\rangle\leq0$. Therefore, \eqref{ithdenkserw} is nonnegative for $\mathcal H^{n-1}$-almost every $y\in\partial K$, and hence 
\begin{equation*}
\int_K \left(y_j^2-p_i y_i^2y_j^2\right)w(y)\,dy \geq0
\end{equation*}
Dividing by	$\int_Kw(y)\,dy$, we obtain
\begin{align*}
m_j-p_i\int_{\R^n} y_i^2y_j^2\,d\mu(y) &\geq 0.
\end{align*}
Therefore
\begin{align*}
(p_i+p_j) \int_{\R^n} y_i^2y_j^2\,d\mu(y) \leq m_i+m_j.
\end{align*}
Applying this inequality and \eqref{csdcrevevre} to \eqref{eqvardecomposunco} completes the proof.
\end{proof}

\section{Appendix A}\label{GaussianCase}
Here we show that, for any $\Sigma\in\mathcal{S}_{\succ0}$ and any $A\in\mathcal{S}$, it holds
\begin{equation}
\operatorname{Var}_{\gamma_\Sigma}\left(\left\langle A\Sigma^{-1}y,y\right\rangle\right)=\int_{\R^n}\left[\left\langle A^2\Sigma^{-1}y,y\right\rangle+\left\langle A\Sigma^{-1}Ay,y\right\rangle\right]\,d\gamma_\Sigma(y).
\end{equation}
For any matrix $M$, since $\gamma_\Sigma$ has covariance matrix $\Sigma$,
\begin{equation}\label{traceIdentity}
\int_{\R^n}\langle My,y\rangle\,d\gamma_{\Sigma}(y)=\operatorname{Tr}(M\Sigma).
\end{equation}
Moreover, if $M$ is symmetric, then
\begin{equation}\label{variationformulaforquaderatic}
\operatorname{Var}_{\gamma_{\Sigma}}(\langle My,y\rangle)=2\operatorname{Tr}(M\Sigma M\Sigma).
\end{equation}
First,
\begin{align*}
\int_{\R^n}\left[\left\langle A^2\Sigma^{-1}y,y\right\rangle+\left\langle A\Sigma^{-1}Ay,y\right\rangle\right]\,d\gamma_\Sigma(y)
&=\operatorname{Tr}\left(A^2\Sigma^{-1}\Sigma\right)+\operatorname{Tr}\left(A\Sigma^{-1}A\Sigma\right)\\
&=\operatorname{Tr}(A^2)+\operatorname{Tr}\left(A\Sigma^{-1}A\Sigma\right).
\end{align*}
On the other hand, by \eqref{variationformulaforquaderatic},
\begin{align*}
\operatorname{Var}_{\gamma_\Sigma}\left(\left\langle A\Sigma^{-1}y,y\right\rangle\right)
&=\operatorname{Var}_{\gamma_\Sigma}\left(\left\langle\frac{A\Sigma^{-1}+\Sigma^{-1}A}{2}y,y\right\rangle\right)\\
&=2\operatorname{Tr}\left[\left(\frac{A\Sigma^{-1}+\Sigma^{-1}A}{2}\right)\Sigma\left(\frac{A\Sigma^{-1}+\Sigma^{-1}A}{2}\right)\Sigma\right]\\
&=\frac12\operatorname{Tr}\left[\left(A+\Sigma^{-1}A\Sigma\right)^2\right]\\
&=\operatorname{Tr}(A^2)+\operatorname{Tr}\left(A\Sigma^{-1}A\Sigma\right).
\end{align*}
Thus equality holds in \eqref{VarlogBM1} when $K=\R^n$.

\section{Appendix B}
We prove the following $\Sigma$-Gaussian analogue of \cite[Fact~8]{CR20}. The same proof applies, and we include it here for completeness.
\begin{lemma}\label{lem:thin-box-convergence}
Fix an orthonormal basis $u_1,\ldots,u_n$ of $\R^n$. For $\varepsilon>0$, let
\begin{equation*}
K_\varepsilon 
=\left\{ t u_1+\sum_{j=2}^n s_j u_j: |t|\leq\frac1\varepsilon,\quad |s_j|\leq\varepsilon \right\}.
\end{equation*}
Let $\Sigma\in\mathcal S_{\succ0}$. If $h\in C^1(\R^n)$ satisfies, for some $C>0$, $|h(x)|\leq Ce^{C|x|}$ and $|\nabla h(x)|\leq Ce^{C|x|}$ for all $x\in\R^n$, then
\begin{equation*}
\lim_{\varepsilon\to0} \int_{\R^n}h(x)\,d\mu_{K_\varepsilon,\Sigma}(x)
=\int_{\R}h(tu_1)\,d\nu_{u_1}(t),
\end{equation*}
where $\nu_{u_1} = N\left( 0,\, \frac{1}{\langle\Sigma^{-1}u_1,u_1\rangle} \right)$.
\end{lemma}

\begin{proof}
Write $x=(t,y)$ in the orthonormal basis $u_1,\ldots,u_n$, where $t\in\R$ is the coordinate in the $u_1$ direction and $y=(y_2,\ldots,y_n)\in\R^{n-1}$ denotes the remaining coordinates in the directions $u_2,\ldots,u_n$. Write
\begin{equation*}
K_\varepsilon = \left\{ tu_1+y:\ |t|\leq\frac1\varepsilon,\ y\in Q_\varepsilon \right\},
\end{equation*}
where
\begin{equation*}
Q_\varepsilon
=\left\{ \sum_{j=2}^n s_j u_j:\ |s_j|\leq\varepsilon \right\}.
\end{equation*}
Define
\begin{align*}
F(\varepsilon)
&:=\int h\,d\mu_{K_\varepsilon,\Sigma}= \frac{ \displaystyle \int_{-1/\varepsilon}^{1/\varepsilon} \int_{Q_\varepsilon} h(tu_1+y) e^{-\frac12\langle\Sigma^{-1}(tu_1+y),tu_1+y\rangle} \,dy\,dt}{\displaystyle
\int_{-1/\varepsilon}^{1/\varepsilon} \int_{Q_\varepsilon} e^{-\frac12\langle\Sigma^{-1}(tu_1+y),tu_1+y\rangle} \,dy\,dt},
\end{align*}
and
\begin{align*}
G(\varepsilon)
:&=\frac{ \displaystyle \int_{-1/\varepsilon}^{1/\varepsilon} \int_{Q_\varepsilon} h(tu_1) e^{-\frac12\langle\Sigma^{-1}(tu_1+y),tu_1+y\rangle} \,dy\,dt }{ \displaystyle \int_{-1/\varepsilon}^{1/\varepsilon}
\int_{Q_\varepsilon} e^{-\frac12\langle\Sigma^{-1}(tu_1+y),tu_1+y\rangle} \,dy\,dt}\\
&=\frac{ \displaystyle \int_{-1/\varepsilon}^{1/\varepsilon} \int_{[-1,1]^{n-1}} h(tu_1) e^{-\frac12 \left\langle \Sigma^{-1} \left(tu_1+\varepsilon\sum_{j=2}^n z_j u_j\right), tu_1+\varepsilon\sum_{j=2}^n z_j u_j \right\rangle} \,dz\,dt }{ \displaystyle \int_{-1/\varepsilon}^{1/\varepsilon} \int_{[-1,1]^{n-1}} e^{-\frac12 \left\langle \Sigma^{-1} \left(tu_1+\varepsilon\sum_{j=2}^n z_j u_j\right),
tu_1+\varepsilon\sum_{j=2}^n z_j u_j \right\rangle} \,dz\,dt }.
\end{align*}
Since $\Sigma^{-1}\succ0$, there exist $c,C>0$ such that
\begin{equation*}
c|x|^2 \leq \langle\Sigma^{-1}x,x\rangle \leq C|x|^2.
\end{equation*}
The growth assumption on $h$, together with the lower bound on $\langle\Sigma^{-1}x,x\rangle$, allows us to apply the dominated convergence theorem to both the numerator and denominator of $G(\varepsilon)$. Hence,
\begin{equation*}
\lim_{\varepsilon\to0^+}G(\varepsilon)
= \frac{ \displaystyle \int_{\R} h(tu_1) e^{-\frac12\langle\Sigma^{-1}u_1,u_1\rangle t^2}\,dt }{ \displaystyle \int_{\R} e^{-\frac12\langle\Sigma^{-1}u_1,u_1\rangle t^2}\,dt }
= \int_{\R}h(tu_1)\,d\nu_{u_1}(t).
\end{equation*}
It remains to show that $|F(\varepsilon)-G(\varepsilon)|\to0$. Let \(0<\varepsilon<1\). If \(tu_1+y\in K_\varepsilon\), then $|y|\leq \sqrt{n-1}\,\varepsilon$. 
Using the growth assumption on \(\nabla h\) along the segment joining \(tu_1\) and \(tu_1+y\), we obtain
\begin{align*}
|h(tu_1+y)-h(tu_1)|
&\leq |y|\sup_{0\leq s\leq1} |\nabla h(tu_1+sy)|\\
&\leq C|y| \sup_{0\leq s\leq1} e^{C|tu_1+sy|}\\
&\leq C\sqrt{n-1}\,\varepsilon\, e^{C(|t|+\sqrt{n-1})}\\
&\leq C_1\varepsilon e^{C_1|t|},
\end{align*}
where \(C_1>0\) is independent of \(\varepsilon\). Hence,
\begin{equation}\label{NumerDenom}
|F(\varepsilon)-G(\varepsilon)| \leq C_1\varepsilon \frac{ \displaystyle \int_{-1/\varepsilon}^{1/\varepsilon} \int_{Q_\varepsilon} e^{C_1|t|} e^{-\frac12\langle\Sigma^{-1}(tu_1+y),tu_1+y\rangle} \,dy\,dt
}{ \displaystyle \int_{-1/\varepsilon}^{1/\varepsilon} \int_{Q_\varepsilon} e^{-\frac12\langle\Sigma^{-1}(tu_1+y),tu_1+y\rangle} \,dy\,dt }.
\end{equation}
Using the lower bound on $\langle\Sigma^{-1}x,x\rangle$, the numerator in \eqref{NumerDenom} is bounded above by
\begin{align*}
(2\varepsilon)^{n-1} \int_{\R} e^{C_1|t|-\frac c2t^2}\,dt =: A\varepsilon^{n-1},
\end{align*}
where $A>0$ is independent of $\varepsilon$. On the other hand, using the upper bound on $\langle\Sigma^{-1}x,x\rangle$ and restricting the integral to $|t|\leq1$, we obtain, for $0<\varepsilon<1$, that the denominator in \eqref{NumerDenom} is bounded below by 
\begin{align*}
(2\varepsilon)^{n-1} \int_{-1}^1 e^{-\frac C2(t^2+n-1)}\,dt =:B\varepsilon^{n-1}
\end{align*}
where $B>0$ is independent of $\varepsilon$. Consequently, as $\varepsilon\to0$
\begin{equation*}
|F(\varepsilon)-G(\varepsilon)| \leq C_1\varepsilon\frac AB \longrightarrow0. 
\end{equation*}
Therefore $\lim_{\varepsilon\to0^+}F(\varepsilon)=\lim_{\varepsilon\to0^+}G(\varepsilon)=\int_{\R}h(tu_1)\,d\nu_{u_1}(t)$, which proves the lemma.
\end{proof}

{\textbf{Acknowledgements.}} The author thanks K\'aroly B\"or\"oczky and Apostolos Giannopoulos for helpful comments that improved the manuscript. AI assistance was provided for the use of the divergence theorem on page 12, a shorter proof of the author’s initial argument for $(ii)\Rightarrow(i)$ on page 8, and Remark \ref{closeTOaxis}.

\bigskip

\noindent
\textsc{Pavlos Kalantzopoulos}\\
Department of Mathematics, University of Waterloo, Canada\\
\textit{Email address:} \texttt{pkalantzopoulos@uwaterloo.ca}


\begin{thebibliography}{999}

\bibitem{BLYZ12}
K.~J.~B\"or\"oczky, E.~Lutwak, D.~Yang, and G.~Zhang:
{\em The log-Brunn--Minkowski inequality}.
Advances in Mathematic 231 (2012), no. 3, 1974--1997.

\bibitem{BK22}
K.J. B\"or\"oczky, P. Kalantzopoulos:
{\em Log-Brunn--Minkowski inequality under symmetry.}
Transactions of the American Society, 375 (2022), no. 8, 5987--6013.

\bibitem{CHLL20}
S. Chen, Y. Huang, Q.-R. Li, J. Liu:
{\em The $L_p$-Brunn-Minkowski inequality for $p<1$.}
Advances in Mathematics, 368 (2020), 107166.

\bibitem{Colesanti08}
A.~Colesanti:
{\em From the Brunn--Minkowski inequality to a class of Poincar\'e-type inequalities}.
Commun.\ Contemp.\ Math.\ 10 (2008), 765--772.

\bibitem{CLM17}
A. Colesanti, G. Livshyts, A. Marsiglietti:
{\em On the stability of Brunn-Minkowski type inequalities.}
J. Functional Analysis, 273 (2017), no. 3, 1120--1139.

\bibitem{CFM04}
D.~Cordero-Erausquin, M.~Fradelizi and B.~Maurey:
{\em The (B) conjecture for the Gaussian measure of dilates of symmetric convex sets and related problems}.
J.\ Funct.\ Anal.\ 214 (2004), 410--427.

\bibitem{CR20}
D.~Cordero-Erausquin and L.~Rotem:
{\em Several results regarding the (B)-conjecture}.
In: Geometric Aspects of Functional Analysis,
Lecture Notes in Math.\ 2256, Springer (2020), 247--262.

\bibitem{CR23}
D.~Cordero-Erausquin and L.~Rotem:
{\em Improved log-concavity for rotationally invariant measures of symmetric convex sets}.
Annals of Prob. 51 (2023), 987--1003.

\bibitem{EsM21}
A. Eskenazis, G. Moschidis:
{\em The dimensional Brunn-Minkowski inequality in Gauss space.}
Journal of Functional Analysis, 280 (2021), Paper No. 108914, 19 pp.

\bibitem{Iffland26}
L. Iffland:
{\em The local logarithmic Brunn--Minkowski inequality for bodies of revolution.}
arXiv:2602.17912 (2026).

\bibitem{KoM22}
A.V. Kolesnikov, E. Milman:
{\em Local $L_p$-Brunn--Minkowski inequalities for $p<1$.}
Mem. Amer. Math. Soc., 277 (2022), no. 1360.

\bibitem{KL21}
A.~V.~Kolesnikov and G.~V.~Livshyts:
{\em On the Gardner--Zvavitch conjecture: symmetry in inequalities of Brunn--Minkowski type}.
Adv.\ Math.\ 384 (2021), 107689.

\bibitem{KLR25}
A.~V.~Kolesnikov, G.~Livshyts and L.~Rotem:
{\em On $p$-Brunn--Minkowski and Brascamp--Lieb inequalities}.
arXiv:2507.12099 (2025).

\bibitem{Lu26}
X. Lu:
{\em Logarithmic Brunn--Minkowski inequality under $n-2$ reflection symmetries.}
arXiv:2608.20730 (2026).

\bibitem{Milman25}
E. Milman:
{\em Centro-affine differential geometry and the log-Minkowski problem.}
J. Eur. Math. Soc., 27 (2025), 709--772.

\bibitem{Put21}
E. Putterman:
{\em Equivalence of the local and global versions of the $L_p$-Brunn-Minkowski inequality.}
Journal of Functional Analysis, 280 (2021), 108956.

\bibitem{Put26}
E. Putterman:
{\em Sharp stability for the $(B)$-theorem.}
arXiv:2608.08472 (2026).

\bibitem{Rotem14}
L. Rotem:
{\em A letter: The log-Brunn--Minkowski inequality for complex bodies.}
arXiv:1412.5321 (2014).

\bibitem{Saroglou14}
{\em C.~Saroglou}:
{\em More on logarithmic sums of convex bodies}.
Mathematika 62 (2016), 818--841.

\bibitem{Saroglou15}
C. Saroglou:
{\em Remarks on the conjectured log-Brunn-Minkowski inequality.}
Geom. Dedicata, 177 (2015), 353--365.

\bibitem{vanHandel23}
R. van Handel:
{\em The local logarithmic Brunn--Minkowski inequality for zonoids.}
Geometric Aspects of Functional Analysis,
Lecture Notes in Math. 2327, Springer (2023), 355--379.

\bibitem{YZ19}
Y.~Yang and D.~Zhang:
{\em The log-Brunn--Minkowski inequality in $\R^3$}.
Proc.\ Amer.\ Math.\ Soc.\ 147 (2019), no.~10, 4465--4475.


\end{thebibliography}
\end{document}